\documentclass{amsart}
\usepackage{amssymb,latexsym}
\usepackage{a4}
\newtheorem{Thm}{Theorem}[section]
\newtheorem{Lemma}[Thm]{Lemma}
\newtheorem{Prop}[Thm]{Proposition}

\newtheorem{Def}[Thm]{Definition}

\def\qed{~\hfill$\square$\medbreak}

\def\N{{\mathbb N}}
\def\Z{{\mathbb Z}}
\def\R{{\mathbb R}}
\def\P{{\mathbb P}}

\def\fa{{\mathfrak a}}
\def\fg{{\mathfrak g}}
\def\fh{{\mathfrak h}}
\def\fk{{\mathfrak k}}
\def\fm{{\mathfrak m}}
\def\fn{{\mathfrak n}}
\def\fp{{\mathfrak p}}
\def\fq{{\mathfrak q}}

\def\cV{\mathcal V}
\def\cA{\mathcal A}
\def\cC{\mathcal C}
\def\cU{\mathcal U}
\def\D{\mathbb D}

\def\sh{\sinh}
\def\ch{\cosh}
\def\es{e^s}
\def\e-s{e^{-s}}
\def\des{e^{2s}}
\def\de-s{e^{-2s}}

\def\SO{\mathrm{SO}}

\def\half{\tfrac12}

\usepackage[usenames]{color}

\begin{document}
\title[Cuspidal discrete series]
{Cuspidal discrete series\\for semisimple symmetric spaces}
\author[Andersen, Flensted-Jensen
and Schlichtkrull]{Nils Byrial Andersen, Mogens Flensted--Jensen\\
and Henrik Schlichtkrull}
\address{Department of Mathematics,
Aarhus University,
Ny Munkegade 118,
Building 1530,
DK-8000 Aarhus C,
Denmark}
\email{byrial@imf.au.dk}
\address{Department of Mathematical Sciences,
University of Copenhagen,
Universitetsparken 5,
DK-2100 Copenhagen \O,
Denmark}
\email{mfj@life.ku.dk, schlicht@math.ku.dk}
\begin{abstract}
We propose a notion of cusp forms on semisimple symmetric spaces.
We then study the real hyperbolic spaces in detail, and show that there exists both
cuspidal and non-cuspidal discrete series. In particular, we show that all the spherical discrete series are
non-cuspidal.
\end{abstract}
\maketitle

\section{Introduction}

The main purpose of this paper is to initiate a
generalization of Harish-Chandra's notion of cusp forms for a real
semisimple Lie group $G$ to the more general case of a semisimple symmetric
space $G/H$. In Harish-Chandra's work on the Plancherel formula for $G$
the fact that all discrete series are cuspidal plays an important role.
However, in the established generalizations to $G/H$
(see \cite{FJ}, \cite{O-M}, \cite{Del}, \cite{B-SI},
\cite{B-SII}), cuspidality plays no role and,
in fact, is not defined at all.

We propose a notion of cuspidal discrete series for
semisimple symmetric spaces $G/H$ in general,
and we
show by explicit calculations on the real hyperbolic spaces
$\SO(p,q+1)_e/\SO(p,q)_e$
that the notion is meaningful in that case.
The notion agrees with the standard one of
Harish-Chandra for the discrete series of $G$, but
in contrast to the situation for $G$, it is not true in general
that all discrete series are cuspidal. Our main result
determines exactly which discrete series representations
for $\SO(p,q+1)_e/\SO(p,q)_e$ are cuspidal. If $p\ge q-1$, all
discrete series representations
are cuspidal, but if $p<q-1$, there is a non-empty and finite family
of non-cuspidal discrete series.

The notion of cuspidality relates to integral geometry on the
symmetric space by using integration over a certain unipotent
subgroup $N^*\subset G$. The definition of $N^*$ is given
in Section \ref{general theory}.
The map $f\mapsto \int_{N^*} f(\cdot nH)\,dn$,
which maps functions on $G/H$ to functions on $G/N^*$,
is a kind of {\it Radon transform} for $G/H$.
A discrete series subspace of $L^2(G/H)$
is said to be {\it cuspidal} if it is annihilated
by this transform (assuming the convergence of the integral
on an appropriate dense subspace of $L^2(G/H)$).
In the group case $G\simeq G\times G/G$,
we have $N^*=N\times\{e\}$, where $N$ corresponds to a
minimal parabolic subgroup of $G$, and thus the Radon transform
of a function $f$ on $G$ is the function $\int_N f(xny^{-1})\,dn$
on $G\times G$. It follows that the annihilation by this
transform agrees with Harish-Chandra's cuspidality condition
for the minimal parabolic subgroup.

It is clear that certain discrete series for $G/H$, which are spherical
(that is, they contain the trivial $K$-type),
cannot be cuspidal since they contain functions
taking only positive values.
Obviously a positive function cannot be annihilated by
integration over any subgroup $N^*\subset G$.
The present investigation shows that for the hyperbolic spaces
all spherical discrete series are non-cuspidal, but also
that in general there exist non-cuspidal, non-spherical
discrete series. The non-spherical non-cuspidal discrete series
are given by odd functions on the real hyperbolic space,
which means that they do not descend to functions on the
projective hyperbolic space.

The first section of the paper describes in more detail
the suggested program for
general symmetric spaces and motivates our study of the
hyperbolic spaces. The hyperbolic spaces are treated
in the following sections. Apart from the motivation,
this treatment is to a large extend independent of the general theory.
The definition of $N^*$ and the generalized notion of cuspidality
was introduced by the second author in lectures at
Oberwolfach (2001).
The results presented in this paper were announced in
\cite{OW}.

\bigskip

\section{A general notion of cuspidality}\label{general theory}

We first recall from Harish-Chandra \cite{HC} the notion of cuspidality
for $G$ and its relation to the Plancherel decomposition.
Let $G$ be a connected semisimple real Lie group with finite center
(or more generally, reductive of Harish-Chandra's class),
and let $K\subset G$ be a maximal compact subgroup with
corresponding Cartan involution $\theta$. Let
$\fg=\fk\oplus\fp$ denote the
corresponding decomposition of the Lie algebra,
and let $\fa\subset\fp$ be a maximal abelian subspace.

Let $\cC(G)$ denote the Schwartz space for $G$, which is
dense in $L^2(G)$.
By definition, a {\it cusp form} on $G$ is a function $f\in\cC(G)$
such that
\begin{equation}\label{HC-cusp}
\int_N f(xny)\,dn=0,
\end{equation}
for all parabolic subgroups $P=MAN\subsetneq G$,
and all $x,y\in G$ (the integral converges absolutely for all
$f\in\cC(G)$). Let $L^2_{\rm ds}(G)$ denote the sum of all
the discrete series representations in $L^2(G)$. It is both left and
right invariant, and the intersection
$\cC_{\rm ds}(G)=L^2_{\rm ds}(G)\cap\cC(G)$ is a dense
subspace.

\begin{Thm}[Harish-Chandra] $\cC_{\rm ds}(G)$
is exactly the space of cusp forms. It is non-zero if and only if
$G$ and $K$ have equal rank.
\end{Thm}

The Plancherel decomposition
splits $L^2(G)$ into a finite sum of {\it series},
each of which is related to a particular
cuspidal parabolic subgroup $P=MAN$
(that is, ${\rm rank}\, M={\rm rank}\, M\cap K$).
The splitting can be accomplished as follows.

Let $\fh_1,\dots,\fh_r$ be a complete (up to conjugation)
set of $\theta$-stable Cartan subalgebras in $\fg$, and let
$\fa_i=\fh_i\cap\fp$ for $i=1,\dots,r$. For each $i=1,\dots,r$,
let $P_i$ be a parabolic subgroup with Langlands decomposition
$M_iA_iN_i$ such that $A_i=\exp\fa_i$.
We can arrange that $\fa_1=\fa$, then $P_1$ is a minimal
parabolic subgroup.

We now define $\cC_i(G)\subset \cC(G)$ as the set of
functions $f\in\cC(G)$ for which
\begin{itemize}
\item $f$ is orthogonal to all $h\in\cC(G)$ with
$\int_{N_i} h(xny)\,dn=0$, for all $x,y\in G$.
\item $\int_{N} f(xny)=0$, for all $x,y\in G$, for all
cuspidal parabolic subgroups
some conjugate of which is properly contained in $P_i$.
\end{itemize}

In particular, for $i=1$ the second condition is vacuous,
and $\cC_1(G)$, which is called the {\it most-continuous}
series, is just the orthocomplement of space of
functions annihilated by all integrals
$\int_{N_1} g(xny)\,dn$.
On the other hand, for ${\rm rank}\, G={\rm rank}\, K$,
we can arrange that $\fa_r=\{0\}$ and $N_r=\{e\}$.
Then for $i=r$ the first
condition is vacuous, and $\cC_r(G)$ is
the space $\cC_{\rm ds}(G)$ of cusp forms.

\begin{Thm}[Harish-Chandra] The following is an orthogonal direct sum
$$\cC(G)=\oplus_{i=1}^r \,\cC_i(G).$$
\end{Thm}

In Harish-Chandra's Plancherel decomposition, each piece
$\cC_i(G)$ (or its closure in $L^2(G)$) is further decomposed
into generalized principal series representations induced from $P_i$.

Let now $G/H$ be a semisimple symmetric space,
that is, the homogeneous space of $G$ with
a subgroup $H$ satisfying
$G^\sigma_e\subset H\subset G^\sigma$, where
$\sigma\colon G\to G$ is an involution, $G^\sigma$ the group of
its fixed points, and $G^\sigma_e$ the identity component of this
group. The problem of obtaining the Plancherel decomposition for $L^2(G/H)$
has been solved (see the references cited in the introduction).
In general terms the outcome is similar to what was described above for
$L^2(G)$. In particular, discrete series occur if and only if
$G/H$ and $K/K\cap H$ have equal rank.

One can also define a Schwartz space $\cC(G/H)$ for $G/H$,
and again (see \cite{B-SI}, Theorem 23.1) there is a finite decomposition
\begin{equation}\label{decomposition of C}
\cC(G/H)=\oplus_i\,\cC_i(G/H)
\end{equation}
where each piece decomposes as a direct integral of
representations induced from a particular parabolic subgroup.
However, the pieces in this decomposition are defined
representation theoretically.
The motivation behind this paper was to study the following
problem.

\smallskip
\noindent{\bf Question:} {\it Is there a description of the $\cC_i(G/H)$
through integrals over subgroups $N$ (or $N/N\cap H$), similar to that
for $G$? In particular, can the discrete
series be characterized through some
reasonable definition of cusp forms?}

\medskip
Recall that a {\it minimal} $\sigma\theta$-stable parabolic subgroup
$P_{\sigma-\rm min}$ is obtained as follows. Let $\fg=\fh\oplus\fq$ be the
decomposition according to $\sigma$. We
may assume that $\sigma$ and $\theta$ commute, and can arrange
that $\fa$ is $\sigma$-invariant and that
$\fa_q:=\fa\cap\fq$ is maximal abelian in $\fp\cap\fq$.
The set $\Sigma$ of non-zero weights of $\fa_q$ in $\fg$
is a root system, and $P_{\sigma-\rm min}=MA_qN$ is
determined from a choice $\Sigma^+$ of positive roots. Here
$A_q=\exp\fa_q$, and $MA_q$ is its centralizer. The
{\it most-continuous series} for $G/H$, which is a basic
summand in (\ref{decomposition of C}),
is induced from a
parabolic subgroup of this form (see \cite{BSmc}).
More generally, the representations in
$\cC_i(G/H)$ are induced from a (not necessarily minimal)
$\sigma\theta$-stable parabolic subgroup $P_i$
(see \cite{B-SII}, Theorem 10.9).

It would be tempting
to apply the unipotent radicals $N_i$ of the $P_i$
in a definition of cusp forms on $G/H$:
\begin{equation}\label{cusp}
\int_{N_i} f(gnH)\,dn=0\quad (g\in G, P_i\neq G).
\end{equation}
In the group case, where $G$ is considered as
a symmetric space for $G\times G$, the $\sigma\theta$-stable
parabolic subgroups of $G\times G$ are of the form $P\times \theta(P)$, where
$P\subset G$ is parabolic, and thus the
integral (\ref{cusp}) becomes an integral over both $N$ and
$\theta(N)$. Hence (\ref{cusp}) differs substantially from
Harish-Chandra's definition (\ref{HC-cusp}) in this case.
Furthermore, although (\ref{cusp}) does converge
in the group case (see \cite{Wallach}, Lemma 15.8.1),
this is not the case for general symmetric spaces.
An example is provided below in Lemma \ref{diverging integrals}
(see however \cite{Kr} for the special case of $L^1$-discrete series
for $G/H$).

Based on these observations one is lead to look for
integrals over different subgroups, and
the following approach was suggested by the second author.
We first fix a system of
positive roots for $\fa$ in $\fg$,
such that $\Sigma^+$ consists of the non-zero restrictions
to $\fa_q$. Since $\Sigma^+$ was already given, this
only amounts to a choice of
positive roots for the root system of {\it pure
$\fa_h$-roots}, that is, the roots of $\fa$
which vanish on $\fa_q$.
On $\fa_q$ an ordering  is determined by $\Sigma^+$.
On $\fa_h$ we choose an ordering which is compatible
with the positive pure roots.
More precisely, these orderings
can be attained by choosing elements $X_q\in\fa_q$
and $X_h\in\fa_h$ such that $\alpha(X_q)>0$ for all
$\alpha\in\Sigma^+$, and $\beta(X_h)>0$
for all positive pure $\fa_h$-roots $\beta$.
Furthermore, we request of $X_h$ that $\alpha(X_h)\neq 0$
for all roots of $\fa$ with non-zero $\fa_h$-restriction.
Then $X_q$ and $X_h$
determine the corresponding notions of positivity
for elements in $\fa_q^*$ and $\fa_h^*$. Notice that
the notion which results from the choice of
$X_h$ is in general not unique.

We now define the following subspaces (in fact subalgebras) of the Lie algebra
$\fn$ of $N$:

\smallskip
$\fn_+=\sum_\beta\fg^\beta$, where
$\beta$ is a root  with $\beta|_{\fa_q}>0$
and $\beta|_{\fa_h}>0$.

$\fn_-=\sigma\theta(\fn_+)=\sum_\beta\fg^\beta$, where
$\beta$ is a root  with $\beta|_{\fa_q}>0$
and $\beta|_{\fa_h}<0$.

$\fn_0=\sum_\beta\fg^\beta$, where
$\beta$ is a root  with $\beta|_{\fa_q}>0$
and $\beta|_{\fa_h}=0$.

\smallskip

\noindent Then $\fn=\fn_+\oplus\fn_0\oplus\fn_-$, and
\begin{equation}\label{defin*}
\fn^*:=\fn_+\oplus\fn_0=\sum_{\beta|_{\fa_q}>0,\, \beta|_{\fa_h}\geq0}\fg^\beta
\end{equation}
is a subalgebra. Let $\fn^{**}= \fn_-$ such that
$\fn=\fn^*\oplus \fn^{**}$. Similarly, let $N^*=\exp(\fn^*)$
and $N^{**}=\exp(\fn^{**})$, then $N=N^*N^{**}$.
Notice that $N^*$ intersects trivially with $H$,
since this is the case already for $N$.

The suggestion is to replace $N$ by $N^*$ in (\ref{cusp}) and
consider integrals of the form
\begin{equation}
\label{newcusp}
\int_{N^*} f(gn^*H)\,dn^*
\end{equation}
in a possible definition of cusp forms on $G/H$.
It is easily seen that in the group case, the
integrals (\ref{newcusp}) amount
exactly to those in Harish-Chandra's
original integral (\ref{HC-cusp})
for minimal parabolic subgroups.

It is useful also to view $N^*$ as a quotient
of the nilpotent part of a particular
minimal parabolic subgroup $P_1$
of $G$. For this purpose we define
\begin{equation}\label{n1}
\fn_1=\fn_+\oplus\fn_0\oplus\theta(\fn_-)\oplus\sum_\beta\fg^\beta,
\end{equation}
with the final summation over the positive pure
$\fa_h$-roots.
It follows from the maximality of $\fa_q$, that
the sum $\sum_\beta\fg^\beta$ in (\ref{n1})
is contained in $\fh$. Using this and the fact that
$\sigma\colon\fn_+\to\theta(\fn_-)$ is bijective,
one concludes easily for the Lie algebra
$$\fn_1=\fn^*\oplus(\fn_1\cap\fh).$$
Let $N_1=\exp(\fn_1)$, then the following holds
similarly.

\begin{Lemma} The mapping $(n_1,n_2)\mapsto n_1n_2$
is a diffeomorphism of $N^*\times (N_1\cap H)$ onto $N_1$.
\end{Lemma}

\begin{proof}
The map $(n_1,n_2)\mapsto n_1n_2$ is clearly injective.
By \cite{Helgason DGLSS}, Lemma VI 5.2,
it is a diffeomorphism onto an open subset
$N^* (N_1\cap H)$ of $N_1$ containing $N_0$.
Let $a_t = \exp(t X_h)$. Then $\lim_{t \mapsto \infty} (a_{-t} n_1 a_t) \in N_0$
for all $n_1 \in N_1$, whence $a_{-t} n_1 a_t \in N^* (N_1\cap H)$ for $t$ sufficiently large.
Since both $N^*$ and $(N_1\cap H)$ are normalized by $a_t$, it follows that $n_1 \in N^* (N_1\cap H)$.
Therefore $N_1 = N^* (N_1\cap H)$, and the result follows.
\qed\end{proof}

We thus have
$$N^* \simeq N_1/N_1\cap H.$$
The construction of $N^*$ is motivated
by the observation that
it is the smallest subgroup of $N$ which can be realized
as a quotient of this form for some minimal parabolic subgroup
of $G$ containing $A$.

For a function $f$ on $G/H$, we define its Radon transform
$Rf$ by
\begin{equation}
\label{R-def}
R f(g) := \int _{N^*} f(gn^*H) \,dn^*
=\int_{N_1/N_1\cap H} f(gn_1H) \,d\dot n_1,\qquad (g\in G),
\end{equation}
provided the integral converges.

Let $\cC(G/H)$ denote the Schwartz space for $G/H$
(see \cite{vdB1992}). At the time of the present research
we expected that (\ref{R-def})  would converge for
all $f \in \cC(G/H)$. However, recently it has been suggested
by van den Ban and Kuit (see \cite{BanKuit}, \cite{BanKuitS})
that the definition of $N^*$ needs to be refined. More precisely,
one needs in addition that the
element $X_h$ is so chosen, that the roots of $\fn_+$ have
non-negative inner product with the sum of the positive
pure $\fa_h$-roots (this can always be attained).
For the hyperbolic spaces investigated in the present paper,
this extra condition is always fulfilled.

Assuming  that $Rf$ is well defined for $f\in\cC(G/H)$,
one can define the cuspidal discrete series for $G/H$
to consist of those discrete series for which the corresponding functions
in $\cC(G/H)$ are annihilated by $R$.

We shall need a result about the relation between $R$ and
invariant differential operators on $G/H$.
We let $P_{\sigma-\rm min}=MA_qN$ be as above.
Since $$\fg=\fn\oplus(\fm\cap\fq)\oplus\fa_q\oplus\fh,$$ we can define
a map $$\mu: \D(G/H)\to\D(M/M\cap H)\otimes \D(A_q)$$
by $\mu(D)=T(D_0)$, where
\begin{equation}\label{u-u0}
u-u_0\in \fn\cU(\fg)+\cU(\fg)\fh,
\end{equation}
and $u\in\cU(\fg)^H, u_0\in\cU(\fm)^{M\cap H}\otimes\cU(\fa_q)$
are elements that represent $D$ and $D_0$, and where
$T(D_0)=a^{-\rho}D_0\circ a^{\rho}$
(see for example \cite{vdB1992}, p. 341). The map is independent of the
choice of positive system for $\fa_q$.

Let $\fm_{nc}$ be the ideal in $\fm$ generated by
$\fm\cap\fp$. It follows from
maximality of $\fa_q$ that $\fm_{nc}\subset\fh$.
The complementary ideal $\fm_c$
is contained in $\fk$ and centralizes $\fa$. Let
$M_c\subset M$ be the corresponding analytic subgroup.
Using the decomposition
\begin{equation}\label{dec m}
\fm=\fm_c\oplus\fm_{nc},
\end{equation}
and the fact that $\fm_{nc}\subset\fh$, we see that
$$\D(M/M\cap H)\simeq\D(M_c/M_c\cap H).$$
Therefore, we may as well regard $\mu$ as a map
$$\mu: \D(G/H)\to\D(M_c/M_c\cap H)\otimes \D(A_q).$$

We denote by $\rho$, $\rho^*$, $\rho^{**}$, $\rho_1\in\fa^*$
half the sum of the roots of $\fn$, $\fn^*$, $\fn^{**}$,
and $\fn_1$ respectively (with multiplicities).
Then $\rho=\rho^*+\rho^{**}$ and
$$\rho_1|_{\fa_q}=(\rho^*-\rho^{**})|_{\fa_q}.$$
Let $f$ be a smooth function on $G/H$, such that
the defining integral of $Rf$ allows the application of right
derivatives by all elements from $\cU(\fg)$.

\begin{Lemma}\label{Af eigenfunction}
Let
\begin{equation}
\cA f(ma) := a^{\rho_1} Rf(ma),
\end{equation}
for $m\in M_c$ and $a\in A_q$.
Then
$$
\cA(Df)=\mu(D)\cA f,
$$
for $D\in\D(G/H)$.
\end{Lemma}

\begin{proof}
Notice first that since $M_cA_q$ centralizes $\fa$,
it preserves $N^*$ in the adjoint action. Moreover,
the pull-back of the invariant measure
$dn^*$ by the action of $ma\in M_cA_q$ is
$a^{-2\rho^*}\,dn^*$. It follows that
\begin{equation}\label{Af}
\cA f(ma)=a^{\rho_1}\int_{N^*} f(man^*H)\,dn^*
=a^{-\rho}\int_{N^*} f(n^*maH)\,dn^*.
\end{equation}

Let $u$ and $u_0$ be as above, and note that
as remarked above we may
assume $u_0\in  \cU(\fm_c+\fa_q)$.
We shall prove that (\ref{u-u0}) implies
\begin{equation}
u-u_0\in \fn^*\cU(\fg)+\cU(\fg)\fh,
\end{equation}
from which the desired property of $A(Df)(ma)$ then follows
by application of $u$ from the right in
the last expression in (\ref{Af}).

By Poincar\'e-Birkhoff-Witt $u-u_0$, modulo $\cU(\fg)\fh$,
is a sum of terms of the form $X_1\cdots X_k Y_1\cdots Y_l$
where $X_1,\dots,X_k$ are root vectors in $\fn$, say for roots
$\alpha_1,\dots,\alpha_k$, and $Y_1,\dots,Y_l$ belong to
$(\fm\cap\fq)+\fa_q$.
We arrange that the basis elements $X_i$ for $\fn$
are ordered such that roots of $\fn^{*}$ come first.
Since $u-u_0$ commutes with $\fa_h$, it follows from the uniqueness
of the expression that $\alpha_1+\dots+\alpha_k$
vanishes on $\fa_h$ for all non-zero terms. As the roots of $\fn^{**}$
are strictly negative on some element in $\fa_h$, it follows that
in each non-zero contribution at least one root vector $X_i$ must
belong to $\fn^*$.\qed
\end{proof}

Notice that if $f$ is an eigenfunction of
the Laplace operator $L$ on $G/H $, then it follows from from
Lemma \ref{Af eigenfunction},
that $\cA f$ is an eigenfunction for $\mu(L)$ on $M_cA$
with the same eigenvalue.
The operator $\mu(L)$ is explicitly determined in
\cite{Arkiv paper}, Lemma 5.3. In particular,
if $M_c\subset H$,
it follows that if $Lf=cf$, then
\begin{equation}\label{eigenvalue}
(L_A-\rho^2)(\cA f)=c\cA f,
\end{equation}
on A. Here $L_A$ is the
(Euclidean) Laplace operator on $A$, normalized
compatibly with the normalization of $L$.
When we define $L$ to be the image of the Casimir element
in $\cU(\fg)$,
this means that correspondingly $L_A$ is the image of
the Casimir element in $\cU(\fa)$.

\section{Notation and definitions for real hyperbolic spaces}

In this and the following sections $G=\SO(p,q+1)_e$ denotes the
identity component of $\SO(p,q+1)$ and $H=\SO(p,q)_e$ the identity component
of $\SO(p,q)$, embedded in the upper left corner of $G$
as the stabilizer of
$x_0=e_{p+q+1}\in\R^{p+q+1}$.
Throughout we assume $p,q \ge 1$. Then $G/H$ is a
non-Riemannian symmetric space.
The corresponding involution $\sigma$ of $G$ is
obtained from conjugation by the diagonal matrix
$(1,\dots,1,-1)$. The fixed point group $G^\sigma$ has two components,
$H$ and $Hc$, where $c\in G$ is the diagonal matrix
$(1,\dots,1,-1,-1)$.

It is known that $G/H$ is simply connected except for $q=1$, where $G/H$ has
an infinite-folded covering. This means that for $q=1$ we can get a somewhat
more general result by going to coverings.

The map $G\ni g\mapsto gx_0$
induces an identification of $G/H$
with the real hyperbolic space $X=X_{p,q}$, defined by the equation
$$ x_1^2 + x_2^2 + ... + x_p^2 - x_{p+1}^2 - ... - x_{p+q+1}^2 = -1$$
in $\R^{p+q+1}$. Likewise $G/G^\sigma$ is identified
with the projective real hyperbolic space $\P X$,
in which antipodal points
$x$ and $-x$ are identified.

The group $K = K_1\times K_2= \SO(p) \times \SO(q+1)\subset G$ is a
maximal compact subgroup, of which the corresponding Cartan involution
will be denoted $\theta$.
We define one-parameter abelian subgroups $A=\{a_t\}\subset G$ and
$T=\{k_\theta\}\subset K_2$
by
\[a_{t} =
\left(
\begin{array}
[c]{ccc}%
\cosh t & 0 & \sinh t\\
0 & I_{p+q-1} & 0 \\
\sinh t & 0 & \cosh t
\end{array}
\right) ,
\]
and
\[k_{\theta} =
\left(
\begin{array}
[c]{cccc}%
I_p & 0 & 0& 0 \\
0& \cos\theta & 0 & \sin\theta\\
0& 0 & I_{q-1} & 0 \\
0& -\sin\theta & 0 & \cos\theta
\end{array}
\right) ,
\]
where $I_j$ denotes the identity matrix of size $j$.
Then
\begin{equation}\label{kax0}
k_\theta a_t x_0 =
(\sinh t, 0,\dots,0;\sin\theta\cosh t,0,\dots,0,\cos\theta\cosh t).
\end{equation}
The semicolon, which will be used again later,
indicates the separation of the first
$p$ from the last $q+1$ coordinates. The generalized Cartan
decomposition $G=KAH$ holds and gives rise to the use of
{\it polar coordinates} on $X$:
$$K\times \R\ni (k,t)\mapsto ka_tH\in X.$$
In this case we have in addition that $K_1\subset H$ and
$K_2=(K_2\cap H)T(K_2 \cap H)$,
where $K_2\cap H=\SO(q)$ centralizes $A$. Hence
\begin{equation}\label{G-decomposition}
K=(K\cap H)T(K_2\cap H)\quad\text{and}\quad G=(K\cap H)TAH.
\end{equation}

In particular, we shall deal with functions $f$ on $G/H$ which are
$K\cap H$-invariant from the left. It follows that such a function
is uniquely determined by the values
$
f(k_\theta a_tH)
$
for $(\theta,t)\in [0,2\pi]\times\R$. Notice, that the antipodal point
corresponding to $k_\theta a_tH$ is $k_{\theta + \pi} a_{-t}H$.

The $K$-types with a $K\cap H$-fixed vector
are generated by the $K\cap H$-bi-invariant
zonal spherical functions on $K$.
In the present case the zonal spherical
functions on $\SO(q+1)/\SO(q)$ are given by
$$k_\theta\mapsto R_\mu (\cos\theta),$$ when $q>1$,
where $ R_\mu$ is a particular Gegenbauer polynomial
of degree $\mu\in \Z^+$ (i.e. $\mu\in \Z$ and $\mu \ge 0$).
We normalize these by $ R_\mu(1)= 1$, and
note that in particular $ R_1(\cos\theta)= \cos\theta.$
When $q=1$, we also allow the integer
$\mu$ to be negative, and replace $R_\mu (\cos\theta)$
by
$$k_\theta\mapsto e^{i\mu\theta}.$$
It follows that a function $f$ on $G/H$
which is $K$-finite of irreducible type $\mu$
and $K\cap H$-invariant, must be of the form
\begin{equation}\label{expression for f}
f(k_\theta a_t) = R_\mu(\cos \theta) f(a_t), \quad\text{respectively}\quad f(k_\theta a_t)=  e^{i\mu\theta}f(a_t).
\end{equation}

In the study of discrete series on semisimple symmetric spaces
$G/H$, one often needs the following general fact:

\begin{Lemma}\label{generating function}
Let $\cV$ be an irreducible, non-zero closed invariant subspace of
$L^2(G/H)$ or $C^\infty(G/H)$.
Then $\cV$ contains a function $f$ with the properties:
\begin{enumerate}
\item[(a)] $f(e)=1$,
\item[(b)] $f$ is an eigenfunction for the Casimir operator on $G/H$,
\item[(c)] $f$ is $K\cap H$-invariant and of some irreducible $K$-type $\mu$.
\end{enumerate}
\end{Lemma}

Since $\cV$ is irreducible, it is generated by this element $f$.

The number $\mu$ is related to the highest weight of
the $K$-type as follows. Let $T\in\fk$
be the infinitesimal generator of $k_\theta$,
then $i\mu$ is the value of the highest weight on $T$
(with a suitable choice of positive restricted roots for $\fk$).

\subsection{Parametrization of discrete series}\label{subs ds}

We define
\begin{equation}\label{defi rho}
\rho = \half (p+q-1) \quad\text{and}\quad \rho_c = \half (q-1),
\end{equation}
and for $\lambda>0$
\begin{equation*}
\mu_{\lambda} = \lambda +\rho - 2\rho_c.
\end{equation*}

We first consider the case $q>1$.
It is known (see for example \cite{FJ}, Section 8)
that the discrete series for the hyperbolic space $G/H$ is
parametrized by the set of positive numbers $\lambda$
such that $\mu_{\lambda} \in \Z$.
The representations that arise from the construction
for general semisimple symmetric spaces in \cite{FJ} are
exactly those for which $\mu_\lambda\geq 0$. The remaining
discrete series (called `exceptional' in \cite{FJ}),
correspond to those (finitely many) parameters $\lambda>0$ for which
$\mu_\lambda<0$. The exceptional
parameters exist if and only if $q>p+3$. The discrete
series representation
with parameter $\lambda$ descends to the projective space
$\P X$ if and only if $\mu_\lambda$ is even.
For general semisimple symmetric spaces, the full
discrete series (including the exceptional ones)
is described in \cite{O-M}.

For $q=1$, where $\rho_c=0$,
we parametrize the discrete series by $\lambda \in \R\setminus\{0\}$
such that $|\lambda| + \rho \in \Z$.
In this case we have $\mu_\lambda=\lambda+\rho\geq1$ for $\lambda>0$,
whereas for $\lambda<0$ we define
$$\mu_\lambda=\lambda-\rho\leq-1.$$
There are no exceptional discrete series.
We note that for $q=1$ every $\lambda \ne 0$ defines a relative
discrete series for the infinite covering space of $G/H$.

We return to the general situation $q\geq1$, and describe the
discrete series which are spherical, according to \cite{FJ}, Theorem 8.2.
Spherical discrete series exist if and only if $q>p+1$, and in this case,
the representation with parameter $\lambda$ is spherical
if and only if $\mu_\lambda\le 0$ and even.

The discrete series parameter $\lambda$ is related to the eigenvalue of the
Laplace-Beltrami operator $\Delta$ of $G/H$ on the
corresponding representation
space in $L^2(G/H)$. More precisely, we have
$\Delta f=(\lambda^2-\rho^2)f$,
for functions $f$ in this space
(with suitable normalization of~$\Delta$), and we can explicitly describe the
discrete series
by a generating function of the form
(\ref{expression for f}) as follows.

\begin{Prop}\label{generating function1}
Let $\lambda\in\R\setminus\{0\}$ be a discrete series parameter
(if $q>1$, this implies in particular that $\lambda >0$).
The corresponding discrete series representation $T_\lambda$ has
a $K\cap H$-invariant generating function
of the following form
\begin{enumerate}
\item[(i)]
For $q>1$ and $\mu_\lambda\ge 0$,
$$ \psi_\lambda (k_\theta a_t) =
R_{\mu_\lambda}(\cos\theta)\,(\ch t)^{-\lambda - \rho}.$$
For $q=1$ and $\mu_\lambda \in\Z$, or for all $\lambda$ for the relative discrete series for the universal covering,
$$ \psi_\lambda (k_\theta a_t) =
e^{i\mu_\lambda\theta}\,(\ch t)^{-|\lambda| - \rho}.$$
\item[(ii)]
For $q>p+3$, $\mu_\lambda=-n < 0$
and
\begin{enumerate}
\item[(I)]
$n= 2m$ even,
$$ \xi_\lambda (k_\theta a_t) =
P_\lambda(\ch^2t) (\ch t)^{-\lambda-\rho-2m}.
$$
\item[(II)]
$n= 2m - 1$ odd,
$$ \xi_\lambda (k_\theta a_t) =\cos \theta\,\,
P_\lambda(\ch^2t) (\ch t)^{-\lambda-\rho-2m},
$$
\end{enumerate}
where in each case $P_\lambda$ is a polynomial of degree $m$.
\end{enumerate}
\end{Prop}

\begin{proof} The expressions for the generating
functions can be derived from
known explicit formulas with hypergeometric functions
for the $K$-finite functions
on $G/H$, see \cite{RLN}, p.~1864, or
\cite{Faraut}, p.~403.
However, we prefer to give an alternative proof which
relates more directly to general theory.

For (i) we refer to \cite{FJ}, formula (8.11), and for
(ii) we refer to \cite{FJO}, where explicit expressions are
determined for the generating functions of the
exceptional discrete series. It follows from \cite{FJO}, Theorem 5.1,
that the following function generates the discrete series
with parameter $\lambda$:
\begin{enumerate}
\item[(I)]
$n= 2m$ even,
$$ \xi_\lambda (k_\theta a_t) =
\phi_{n,m}(\sh^2 t),
$$
\item[(II)]
$n= 2m - 1$ odd,
$$ \xi_\lambda (k_\theta a_t) =
\cos \theta\ch t \,\,\phi_{n, m}(\sh^2 t).
$$
\end{enumerate}
Here $\phi_{n,m}$ is the function on $\R^+$ defined by
\begin{equation}\label{defi phi}
\phi_{n, m}(s^2)= [\omega^m(1+x^2)^{n-\rho_c}]_{|x=(s, 0, ... ,0)},
\end{equation}
where $x\in\R^p$, and where
$\omega$ denotes the Laplace operator on $\R^p$.

Note that $\xi_\lambda$ differs from the function constructed
in \cite{FJO} by being $K\cap H$-invariant.
In the notation of \cite{FJO}, the $K\cap H$-invariant
generating function is $\int_{K\cap H}\xi^o_{\lambda,\omega^m} (kg)\, dk $.

In order to prove the proposition it now suffices to show
for each relevant pair $(n,m)$ that there exists a polynomial
$P$ of degree $m$ such that
\begin{equation}\label{expression phi}
\phi_{n,m}(s^2)=P(s^2)(1+s^2)^{n-2m-\rho_c},\quad(s\in\R)
\end{equation}
for all $s\in\R$.
The expression (\ref{expression phi}) is derived from (\ref{defi phi})
by successive use of the following lemma. Note that
$n-\rho_c=-\lambda-\rho+\rho_c<-\frac p2.$
\qed\end{proof}

\begin{Lemma}
Let $Q$ be a polynomial of degree $d$ and let $\nu\in\R$. Then
there exists a polynomial $P$ of degree $\le d+1$ such that
$$
\omega(Q(x^2)(1+x^2)^\nu)=P(x^2)(1+x^2)^{\nu-2},\quad(x\in\R^p).
$$
If $\nu+d\neq 0$ and $\nu+d\neq -\frac{p-2}2$ then $\deg P=d+1$
\end{Lemma}

\begin{proof} The existence of the polynomial $P$ is an easy
computation. The statement about its degree follows from
the fact that
$$
Q(x^2)(1+x^2)^\nu=(x^2)^{\nu+d}+\text{ lower order terms,}
$$
since $\omega(x^\mu)=\mu(\mu+p-2)x^{\mu-2}$ for all $\mu\in\R$.\qed
\end{proof}

\section{A unipotent subgroup}

Define
$$
n_{u,v}=\exp(Z_{u,v})\in G, \quad\text{with}\quad
Z_{u,v}=\left(
\begin{array}
[c]{cccc}%
0 & u & v & 0\\
-u^t & 0  & 0 & u^t\\
v^t & 0 & 0 & -v^t \\
0 & u & v & 0
\end{array}
\right)\in \fg,
$$
where $u\in\R^{p-1}$ and $v\in\R^q$ are considered as rows,
and $u^t$, $v^t$ are the corresponding columns.
If $Y$ denotes the infinitesimal generator of $a_t$, then
$[Y,Z_{u,v}]=Z_{u,v}$ for all $u,v$.
The matrices $Z_{u,v}$ span a commutative subalgebra $\fn$ of $\fg$,
and the subgroup $N=\exp(\fn)$
is the unipotent radical of a minimal
$\sigma\theta$-stable parabolic subgroup $P$ in $G$. In particular,
we note that $N\cap H$ is trivial.

Easy calculations show
\begin{equation}\label{nx0}
n_{u,v} x_0 = (\half (u^2 - v^2),u;{-}v,1 + \half (u^2 - v^2)),
\end{equation}
and
\begin{equation}\label{anx0}
a_s n_{u,v} x_0 = (\sh s + \half \es (u^2 - v^2),u;{-}v
,\ch s + \half \es (u^2 - v^2)),
\end{equation}
for all $s\in\R$, where $u^2=u\cdot u$ and $v^2=v\cdot v$ as usual.

\begin{Lemma}\label{diverging integrals}
Assume $p>1$ and
$p+q>3$. Then there exists a non-negative $K$-invariant
Schwartz function $f\in\cC(G/H)$ for which the integral
$\int_N f(n)\,dn$ diverges.

Assume in addition
$q>p+1$. Then the integral diverges for the
$K$-invariant discrete series function $f=\psi_\lambda$
where $\lambda=\frac12(q-p-1)$ (see Proposition
\ref{generating function1}(i)).
\end{Lemma}

\begin{proof} Let $f(ka_tH)=(\cosh t)^{-\rho-\nu}$, with $\nu>0$, then
$f\in \cC(G/H)$ (see the remark after Theorem \ref{first prop}).
Using (\ref{kax0}) and (\ref{nx0}), we find
$$
\begin{aligned}
\int_N f(n)\,dn&=\int_{\R^q} \int_{\R^{p-1}} f(n_{u,v})\, du\, dv\\
&=\int_{\R^q} \int_{\R^{p-1}} (v^2+(1+\frac12(u^2-v^2))^2)^{-\frac12(\rho+\nu)}\, du\, dv\\
&=\int_0^\infty \int_0^\infty (y^2+(1+\frac12(x^2-y^2))^2)^{-\frac12(\rho+\nu)}
x^{p-2} y^{q-1}\, dx\, dy.
\end{aligned}
$$
In particular, for $1\le y\le x\le y+1$, we have
$0\le x^2-y^2=(x-y)(x+y)\le 2y+1$, and hence
$$y^2+(1+\frac12(x^2-y^2))^2\le 10y^2.$$
Hence
$$
\begin{aligned}
\int_N f(n)\,dn
&\geq \int_1^\infty \int_{y}^{y+1}
(10y^2)^{-\frac12(\rho+\nu)} x^{p-2} y^{q-1}\, dx\, dy\\
&\geq C\int_1^\infty
y^{-(\rho+\nu)+p+q-3}\, dy,
\end{aligned}
$$
with $C>0$. This integral
diverges when $\nu\le \frac12(p+q-3)$.

If $q>p+1$ and $\lambda=\frac12(q-p-1)$, then the function $f$ defined above
with $\nu=\lambda$ is exactly $\psi_\lambda$. The integral diverges
since in this case $\lambda=\frac12(q-p-1)\le\frac12(p+q-3)$.
\qed\end{proof}

If $p+q\le 3$, it is likely that the $N$-integral converges for all
Schwartz functions, but we shall not consider this question here.

Motivated by Section \ref{general theory},
we now define the following subgroup of $N$.
Let
$$u=(u_1,\dots,u_{p-1})\in \R^{p-1}\quad\text{and}\quad
v = (v_q, \dots ,v_1)\in\R^q.$$
It is convenient to number the entries of $v$ from right to left
as indicated. It is not difficult to verify that the following agrees with
(\ref{defin*}).

\begin{Def} Let $N^*\subset N$ be the $\max (p-1, q)$-dimensional subgroup
$$
N^*=\{ n_{u,v} \mid u\in\R^{p-1}, v\in\R^q, u_j=v_j \text{ for } j=1,\dots,l \},
$$
where $l = \min (p-1, q)$.
\end{Def}

We want to integrate $K\cap H$-invariant functions on $G/H$
over sets of the form $a_sN^*$,
where $s\in\R$, with respect to Haar measure of $N^*$.
For this purpose we shall need the following,
which is easily deduced from
(\ref{kax0}) and (\ref{anx0}).

The relation
$$(K\cap H)k_\theta a_t H = (K\cap H) a_s n_{u,v}H $$
implies
$$\ch^2 t = v^2 + (\ch s + \half\es (u^2 - v^2))^2$$
and
$$ \cos \theta = (\ch s + \half\es (u^2 - v^2))/ \ch t$$
for all $\theta$, $t$, $s$, $u$ and $v$.

When $p = 1$, the value of $t$, including its sign,
can also be determined by
$$ \sh t = \sh s + \half\es (- v^2), $$
whereas if $p > 1$ the sign is redundant since by (\ref{kax0})
the double cosets $(K\cap H)k_\theta a_t H$ and $(K\cap H)k_\theta a_{-t} H$
are identical. We assume in this case that $t\ge 0$.

With these relations between $(s,u,v)$ and $(\theta,t)$,
we have $f(a_s n_{u,v}H)= f(k_\theta a_t H)$
for $K\cap H$-invariant functions $f$ on $G/H$.

We assume now $n_{u,v} \in N^*$. Then the expression $u^2-v^2$ simplifies.
We distinguish the two cases:

\smallskip
{\bf A.} $p> q$. Then $u=(v_1,\dots,v_q,u')$, where
$u'\in\R^{p-1-q}$. We put $x=\|u'\|$ and  $y=\|v\|$, and obtain
\begin{equation}\label{first cosh}
\ch^2 t= y^2+(\ch s+\half\es x^2)^2.
\end{equation}
and
\begin{equation}\label{first cos}
\cos \theta = (\ch s + \half\es x^2)/ \ch t.
\end{equation}

For the integration over $N^*$ we use polar coordinates
for $u'$ and $v$ with
\begin{equation}\label{first alphabeta}
\alpha=p-2-q, \quad \beta=q-1.
\end{equation}

We conclude that the measure on $N^*$ can be normalized such that
for a $K\cap H$-invariant function,
\begin{equation}\label{N*integral}
\int_{N^*} f(a_s n^*H)\,dn^*
=\int_0^\infty  \int_0^\infty f(k_\theta a_tH)\, x^\alpha y^\beta\, dx\, dy,
\end{equation}
where $t=t(s,x,y)\geq 0$ is determined
by (\ref{first cosh}) and $\theta= \theta(s,x,y)$ by (\ref{first cos}).

Note that in the degenerate case $p-1=q$,
we have $u'=0$. Hence $x=0$ in (\ref{first cosh}),
and the right hand side of (\ref{N*integral})
has to be interpreted without the integration over $x$.

\smallskip
{\bf B.} $q \ge p $. Then $v=(v',u_{p-1},\dots,u_1)$, where
$v'\in\R^{q-p+1}$. We put $x=\|u\|$ and  $y=\|v'\|$, and obtain
\begin{equation}\label{second cosh}
\ch^2 t= x^2+y^2+(\ch s-\half\es y^2)^2.
\end{equation}
and
\begin{equation}\label{second cos}
\cos \theta = (\ch s - \half\es y^2)/ \ch t.
\end{equation}

We use polar coordinates
for $u$ and $v'$ with
\begin{equation}\label{second alphabeta}
\alpha=p-2, \quad \beta=q-p.
\end{equation}
Then (\ref{N*integral}) holds with $t=t(s,x,y)\geq 0$ determined
by (\ref{second cosh}) and $\theta= \theta(s,x,y)$ by (\ref{second cos}).

In the degenerate case $p=1$, the sign of $t$ is determined by
$$ \sh t = \sh s - \half\es (y^2). $$
Note also that in this case $u=0$, so that
$x=0$ in (\ref{second cosh}), and again (\ref{N*integral})
is interpreted without integration over $x$.

\medskip
To summarize, let us define
for $s,x,y\in\R$
\begin{equation}\label{def Theta}
\Theta(s,x,y)=\begin{cases}
y^2+(\ch s+\half\es x^2)^2, & (p> q) \\
x^2+y^2+(\ch s-\half\es y^2)^2, \quad& (q \ge p).
\end{cases}
\end{equation}
Then in particular
for a $K$-invariant function on $G/H$, we can write
\begin{equation}\label{N*integral2}
\int_{N^*} f(a_s n^*H)\,dn^* =
\int_0^\infty  \int_0^\infty F(\Theta(s,x,y)) x^\alpha y^\beta\, dx\, dy,
\end{equation}
where $F$ is the function $F(\ch^2 t)=f(a_tH)$ on $\R_+$,
and $\alpha$ and $\beta$ are given by either (\ref{first alphabeta}) or (\ref{second alphabeta}).

We have normalized the measure on $N^*$ such that
this integral equation is valid without a constant.

\section{Main results for real hyperbolic spaces}

For functions on $G/H$ we define, assuming convergence,
$$Rf(g) = \int _{N^*} f(g n^*H)\, dn^*,  \qquad (g\in G),$$
in accordance with the definition of the Radon transform
in Section \ref{general theory}.

We shall be particularly interested in the values of
$Rf$ on the elements $a_s$. For simplicity we write
$$ Rf(s) = Rf(a_s)= \int _{N^*} f(a_s n^*H)\, dn^*,\quad (s\in\R). $$
For $K$-invariant functions this integral is explicitly
expressed in (\ref{N*integral2}).

Referring to Lemma \ref{Af eigenfunction}, we find
$$\rho_1=
\begin{cases}
\frac12(p-q-1)\, & \text{if } p> q \\
\frac12(q-p+1) & \text{if } p\le q,
\end{cases}
$$
and recall that $\cA f(a)=a^{\rho_1} Rf(a)$.
It follows from (\ref{eigenvalue})
that if $Lf=(\lambda^2-\rho^2)f$, then $(d/ds)^2\cA f=\lambda^2\cA f$,
and hence $Rf$ is a linear combination
\begin{equation}\label{exponential function combination}
Rf(s)=C_1e^{(-\rho_1+\lambda)s}+C_2e^{(-\rho_1-\lambda)s}.
\end{equation}

\begin{Thm}\label{first prop}
Let $f$ be a continuous function on $G/H$, and assume
there exists a constant $C>0$, such that
\begin{equation}\label{Schwartz1}
|f(ka_t)| \le C\ (\ch t)^{-\rho}\, (1+\log(\ch t))^{-{2}},
\end{equation}
for all $t\in\R$, $k\in K$.

\begin{enumerate}
\item[(i)] {\rm Convergence.} The defining integral of $Rf(s)$
converges absolutely for all $s\in\R$.

\item[(ii)]{\rm Compact support.} If $f(ka_t) = 0$ for all
$k\in K$ and $|t| \ge t_0 >0 $, then
$$Rf(s) = 0,$$
\begin{description}
\item[A] for all $|s|\ge t_0$ if $p>q$,
\item[B] for all $s\le -t_0$ if $p\le q$.
\end{description}

\item[(iii)]{\rm Decay.} Let $N\ge {2}$, and put
\begin{equation}\label{SchwartzN}
\|f\|_N=\sup_{t\in\R, k\in K} \,(\ch t)^{\rho}\, (1+\log(\ch t))^{N}\,|f(ka_t)|.
\end{equation}
Assume $\|f\|_N <\infty$. There exists a constant $C_N>0$ (independent of $f$), such that
$$e^{\rho_1 s}\, |Rf(s)| \leq C_N \|f\|_N(1+|s|)^{-(N-{2})}, $$
\begin{description}
\item[a] for all $s\in\R$ if $p\ge q$,
\item[b] for all $s\le 0$ if $p< q$.
\end{description}

\item[(iv)]{\rm Limits.} Assume that $p < q$.
The function $\es Rf(s)$ is bounded on $\R^+$.
If $f$ is $K$-invariant, then the limit $\lim_{s\to\infty}\es Rf(s)$
exists, and if in addition $f$ is positive and
not identically zero,
then this limit is positive.

\end{enumerate}
\end{Thm}
Note the difference between the conditions for {\bf A}, {\bf B}
versus {\bf a}, {\bf b}.

We also remark that a Schwartz function $f\in\cC(G/H)$ by definition, see \cite{NBA}, Definition 2.1, satisfies
the growth conditions $\|D f\|_N <\infty$, for all $D\in\D(G/H)$ and all $N\in \N$, with $\|\cdot\|_N$ defined
by (\ref{SchwartzN}). In particular, $f\in\cC(G/H)$ satisfies (\ref{Schwartz1}).

\begin{Thm}\label{second thm}
Let $\lambda \ne 0$ be a discrete series parameter,
and let $f$ be the generating function of
Proposition~\ref{generating function1}.
\begin{enumerate}
\item If $\lambda > 0$ and $\mu_\lambda> 0$, then
$Rf=0$. Likewise if $\lambda <0$ and $\mu_\lambda< 0$.
\item If $\lambda > 0$ and $\mu_\lambda\le 0$, then
$Rf(s) = Ce^{(-\rho_1+\lambda)s}\,(s\in\R)$,
for some $C\neq 0$.
\end{enumerate}
\end{Thm}

Notice that the second statement in (1) is only relevant when $q=1$.
In this case (2) never occurs, and we always have $Rf=0$.
This is also the case for all relative discrete series
parameters (see Subsection \ref{subs ds}).

In conclusion, we have the following characterization of the cuspidal and non-cuspidal discrete series.

\begin{Thm}\label{third thm}
Let $\lambda \ne 0$ be a discrete series parameter.
\begin{enumerate}
\item Let $q>1$. Then the discrete series representation $T_\lambda$
is cuspidal if and only if $\mu_\lambda>0$.
\item Let $q\le p+1$. Then all discrete series are cuspidal
(if $q=1$, then all the relative discrete series
are also cuspidal).
\item All spherical discrete series for G/H are non-cuspidal.
These representations exist if and only if $q>p+1$.
\item There exist non-spherical non-cuspidal discrete series
if and only if $q>p+3$. These representations do not descend to discrete series
of the real projective hyperbolic space.
\end{enumerate}
\end{Thm}

This follows easily from Theorem \ref{second thm} and the
description of the discrete series in Subsection \ref{subs ds}.

\section{Proofs}
The proofs are based on the following two lemmas.

\begin{Lemma}\label{bound Theta}
For each $s\in\R$, there exist numbers $a,b>0$ such that
$$
\Theta (s,x,y) \ge
\begin{cases}
y^2+ax^4+b \quad &\text{if} \quad p>q \\
x^2+ay^4+b \quad &\text{if} \quad q\ge p ,
\end{cases}
$$
for all $x,y\in\R$. If $p>q$ or $s\le 0$, then
the numbers $a,b$ can be chosen as follows
$$
a=\frac14 e^{2s},\, b=\cosh^2s.
$$
\end{Lemma}

\begin{proof} The statements for $p>q$ are straightforward from
(\ref{def Theta}). Assume $p\le q$. From (\ref{def Theta})
we obtain
$$\Theta(s,x,y)=x^2+\frac14e^{2s}y^4+\frac12(1-e^{2s})y^2+\cosh^2s.$$
The statements for $s\leq0$ follow. By rewriting
\begin{equation}\label{Theta}
\Theta(s,x,y)=x^2+
\frac14e^{2s}(y^2-1)^2+\frac12y^2+\frac12+\frac14e^{-2s},
\end{equation}
we finally see that for $s\geq0$
$$\Theta(s,x,y)\ge x^2+
\frac14(y^2-1)^2+\frac12y^2+\frac12=x^2+\frac14y^4+\frac34.\qquad\square$$
\end{proof}

\begin{Lemma}\label{integrabilitet}
Let $a,b,c,d$ and $\gamma$ be $> 0$, then
$$ \int_0^\infty \int_0^\infty (1+x^a+y^b)^{-\gamma} x^{c-1} y^{d-1}\, dx\, dy < \infty,$$
if $$\frac{c}a + \frac{d}b <\gamma .$$
Furthermore
$$ \int_0^\infty \int_0^\infty (1+x^a+y^b)^{-\gamma} x^{c-1} y^{d-1}
(1+\log(1+x^a+y^b))^{-{\delta}}\, dx\, dy < \infty,$$
if {$\delta>1$ and}
$$\frac{c}a + \frac{d}b =\gamma .$$
\end{Lemma}

\begin{proof} Easy. \qed\end{proof}

\subsection*{Proof of Theorem \ref{first prop}}

We may assume that $f$ is $K$-invariant, since otherwise
we can replace it by the continuous function
$gH\mapsto \sup_{k\in K} |f(kgH)|$.
The defining integral (\ref{N*integral2}) of $Rf(s)$
is bounded by the following integral
$$|Rf(s)| \le
\int_0^\infty \int_0^\infty x^\alpha y^\beta\,
\Theta(s,x,y)^{-\rho/2}\,
(1+\log{\Theta(s,x,y))}^{-{2}} \,dx\, dy,$$
with the values of $\alpha$, $\beta$, and $\rho$
from (\ref{first alphabeta}), (\ref{second alphabeta}),
and (\ref{defi rho}). The convergence of this
integral is an easy consequence of the preceding lemmas.
Note that the logarithmic term is
needed for the convergence. This proves (i).

It is seen from Lemma \ref{bound Theta}
that if $p>q$ or $s \le 0$, then
$$
\Theta (s,x,y) \ge \ch^2 s.
$$
We can thus conclude that
if  $f(a_t) = 0$ for all $|t| \ge t_0$, then
the integrand in (\ref{N*integral2}) vanishes
for $|s|\ge t_0$ if $p>q$, and for
$s\le -t_0$ if $p\le q$.
This proves (ii).

We now study the asymptotic behavior of
$Rf(s)$ as $s \to \pm \infty$. The following arguments
have to be adapted slightly in the degenerate cases $p=q+1$ and $p=1$,
where there is no $x$-integral.

Assume first that $p>q$.
Recall that $\alpha=p-2-q$ and $\beta=q-1$, and that
\begin{equation}\label{A integral}
Rf(s) = \int_0^\infty \int_0^\infty
F(\Theta(s,x,y))\, x^\alpha y^\beta\, dx\,dy.
\end{equation}
From Lemma \ref{bound Theta}
$$
\Theta(s,x,y)\ge y^2+\frac14e^{2s}x^4+\ch^2s,
$$
and by the definition of $\|f\|_N$, we have
\begin{equation}\label{upper bound F(Theta)}
|F(\Theta)|\le C \|f\|_N \Theta^{-\rho/2}\, (1+\log(\Theta))^{-N}.
\end{equation}
We insert this bound in (\ref{A integral}),
replace $\Theta(s,x,y)$ by the lower bound, and
substitute $\frac12e^sx^2=\ch(s) \xi^2$ and $y=\cosh(s) \eta$,
so that
\begin{equation}\label{subst result}
y^2+\frac14e^{2s}x^4+\ch^2s=(1+\eta^2+\xi^4)\cosh^2s.
\end{equation}
Simplifying by the relation $\frac12(\alpha+1)+\beta+1=\rho$,
we finally obtain
$$
|Rf(s)|\le C \|f\|_N e^{-\frac12(\alpha+1)s}\int_0^\infty \int_0^\infty
(1+\eta^2+\xi^4)^{-\rho/2}
(1+\log(\cdot))^{-N}
\eta^\beta \xi^\alpha\, d\eta\,d\xi,
$$
where (\ref{subst result}) is the omitted argument in the logarithm.
The logarithmic term is bounded above by a constant times
$$
(1+\log(1+\eta^2+\xi^4))^{-{2}}(1+|s|))^{-(N-{2})}.
$$
With Lemma \ref{integrabilitet} statement (iii)a follows,
except for the case $p=q$.

We now assume $p\le q$. Then $\alpha=p-2$ and $\beta=q-p$.
For $s\leq 0$, we use the estimate
$$
\Theta(s,x,y)\ge x^2+\frac14e^{2s}y^4+\ch^2s,
$$
from Lemma \ref{bound Theta}. Proceeding as before, with
the roles of $x$ and $y$ interchanged, we obtain (iii)b
and the negative direction of the remaining case in (iii)a.

If instead we use the estimate
$$
\Theta(s,x,y)\ge x^2+\frac14e^{2s}(y^2-1)^2,
$$
which follows from (\ref{Theta}), and
substitute $x=\frac12e^{s}|y^2-1|\xi$, we obtain for $s>0$
$$
|Rf(s)|\le C e^{-s/2} (1+s)^{N-{2}}.
$$
Here $C\in]0,\infty]$ is given by
$$
C=\int_0^\infty
(1+\xi^2)^{-\rho/2}\xi^\alpha\,d\xi
\int_0^\infty
|y^2-1|^{(p-q-1)/2}
(1+\log(\cdot))^{-{2}}
 y^\beta\, dy,
$$
where argument in the logarithm is the maximum of $1$ and
$(y^2-1)^2$. The integral is finite precisely when $p=q$, and in this case we
thus obtain the desired bound in the positive direction.
This finishes the proof of (iii).

Assume $p<q$. Recall that $$\Theta(s,x,y)=x^2+y^2+(\ch s-\half\es y^2)^2.$$
Let $v =-\sh s + 1/2\es y^2$, then $y^2 = 1 +2\e-s v - \de-s$ and
$$\Theta(s,x,y)=1+x^2+v^2.$$
With this substitution, we find
\begin{equation}\label{zz}
Rf(s) = \e-s \int_0^\infty \int_{-\sh s}^\infty F(1 + x^2 + v^2)
(1 +2\e-s v - \de-s)^{(\beta -1)/2}  x^\alpha\, dv\, dx,
\end{equation}
with the following upper bound, for $s\ge 0$, since $\beta \ge 1$:
$$
|Rf(s)|\le C \e-s \int_0^\infty \int_{-\infty} ^\infty
\frac{(1 + x^2 + v^2)^{-\rho/2}}
{{[}1+\log(1 + x^2 + v^2){]^2}}
(1+2|v|)^{(\beta -1)/2}  x^\alpha\, dv\, dx < +\infty.
$$
Applying Lebesgue's theorem, we get
$$ \lim_{s\to\infty}\es Rf(s) =
\int_0^\infty \int_{-\infty}^\infty F(1 + x^2 + v^2) x^\alpha dv dx$$
from which (iv) follows. \qed

Before we proceed, we note that if $f$ satisfies a sharper decay
than (\ref{Schwartz1}), we can improve on the decay of $Rf(s)$
expressed in  (iii).

\begin{Lemma}\label{exponential decay lemma}
Let $f$ be a continuous function on $G/H$, and assume that
\begin{equation*}
\sup_{t\in\R, k\in K}\, (\ch t)^{\rho+\gamma} |f(ka_t)| <\infty
\end{equation*}
for some $\gamma>0$. Then for each $\epsilon > 0$
there exists a constant $C>0$, such that
$$ e^{\rho_1s}\, |Rf(s)| \leq C (\ch s)^{-\gamma + \epsilon}, $$
\begin{description}
\item[a]
for all $s\in\R$  if $p\ge q$,
\item[b]
for all $s\le 0$ if $p < q$.
\end{description}
\end{Lemma}

\begin{proof}
Replace (\ref{upper bound F(Theta)}) by
\begin{equation*}
|F(\Theta)|\le C \Theta^{-(\rho+\gamma)/2}
\end{equation*}
in the proof of (iii) above and proceed as before.\qed
\end{proof}

\subsection*{Proof of Theorem \ref{second thm}}

Let $\lambda$ be a discrete series parameter, and let
$f$ be the generating function of Proposition
\ref{generating function1}.
We know, see (\ref{exponential function combination}), that
$s\mapsto Rf(ka_s)$ is a linear combination of
$e^{(-\rho_1 +\lambda)s} $ and $e^{(-\rho_1 -\lambda) s}$
for each $k\in K$.

We shall first establish conclusion
(1) in the theorem. Here
we use the following lemma.
The vanishing of $Rf$ will be established by showing that,
for some $\alpha\in\R$ and all $k\in K$,
$e^{\alpha s}Rf(ka_s)$ decays to $0$
as $s\to -\infty$, and that it is bounded in the other direction.

\begin{Lemma} Let $\phi\colon\R\to\R$ be a linear combination
of exponential functions with real exponents.
If $\phi(s)$ is bounded as $s\to\infty$,
and decays to $0$ as $s\to -\infty$, then $\phi=0$.
\end{Lemma}

\begin{proof} The maximal exponent is $\le 0$
and the minimal exponent is $>0$. \qed\end{proof}

For simplicity we assume in what follows that $p>1$.
The arguments for the case $p=1$ are similar.
Notice that if $q=1$ and $\lambda < 0$, then by definition
$\psi_\lambda$ equals the complex conjugate of
$\psi_{-\lambda}$. Hence $R\psi_\lambda$ is the complex
conjugate of $R\psi_{-\lambda}$. Therefore, to prove
Theorem \ref{second thm}(1), we may assume $\lambda>0$.

We first assume $\mu_\lambda > 0$.
It follows from the expression for $f=\psi_\lambda$
in Proposition \ref{generating function1}(i) that
$$|f(k a_t)|\le C (\ch t)^{-\lambda-\rho},$$
for all $k\in K$ and $t\in \R$.

For $p\ge q$, it follows from Theorem \ref{first prop} (iii)a
that
$$e^{\rho_1s}Rf(ka_s)\to 0$$ in both directions $s\to\pm\infty$,
thus $Rf(ka_s)=0$, for all $s\in\R$ and $k\in K$, and hence
$Rf=0$.

For $p<q$, it follows from Lemma \ref{exponential decay lemma} that
$e^{(\rho_1-\lambda+\epsilon)s}Rf(ka_s)$ is bounded for $s\to-\infty$,
for all $\epsilon>0$.
Since $\rho_1-\lambda=-\mu_\lambda+1<1$, we can choose
$\epsilon=\frac12(1-\rho_1+\lambda)$, and conclude that
$$\es Rf(ka_s)=e^{\epsilon s}e^{(\rho_1-\lambda+\epsilon)s}Rf(ka_s)
\to 0$$ for $s\to-\infty$.
On the other hand, from Theorem
\ref{first prop} (iv),
we infer that $\es Rf(ka_s)$ is bounded at $+\infty$. Thus again $Rf=0$.
This proves Theorem \ref{second thm}(1).

Assume now that $\mu_\lambda =0$. This only happens when $p+1<q$.
From Lemma \ref{exponential decay lemma}, we see
that $e^{\rho_1 s}Rf(s)\to 0$ as $s\to -\infty$.
We know that $e^{\rho_1 s}Rf(s)$ is a
linear combination of the exponential
functions $e^{\lambda s}$ and $e^{-\lambda s}$,
but because of the limit relation at $-\infty$,
only the first one can occur. Hence $Rf(s)$
is a multiple of $e^{(-\rho_1+\lambda)s}=\e-s$.
Since $f\ge 0$, we conclude from
Theorem \ref{first prop}(iv) that
$\lim_{s\to\infty} \es Rf(s)=C\neq 0$.
Hence $Rf(s)=C\e-s$.

Assume finally that $\mu_\lambda <0$, then $p+3<q$ and $f=\xi_\lambda$
in Proposition \ref{generating function1}(ii).
From the expressions in (ii), we infer immediately
\begin{equation}\label{bound xi}
|\xi_\lambda(k_\theta a_t)| \le C  (\ch t)^{-(\lambda+\rho)}.
\end{equation}
Furthermore, since $P_\lambda $ has degree $m$, the limit
\begin{equation}\label{limit xi}
\lim_{t\to\infty}(\ch t)^{\lambda+\rho} \xi_\lambda(a_t)
\end{equation}
exists and is non-zero.

We first consider the even case $n=2m$, where $\xi_\lambda$
is $K$-spherical. We shall need the following lemma.

\begin{Lemma}\label{exponential limit lemma}
Assume $p<q$.
Let $f$ be a $K$-invariant continuous function on $G/H$, and assume that
there exists $\gamma>0$ such that
\begin{equation*}
\sup_{t\in\R}\, (\ch t)^{\rho+\gamma} |f(a_t)| <\infty,
\end{equation*}
and such that
\begin{equation*}
\lim_{t\to \infty}\, (\ch t)^{\rho+\gamma} f(a_t)
\end{equation*}
exists and is non-zero. Then
$$ \lim_{s\to -\infty} e^{(\rho_1-\gamma)s}\, Rf(s) $$
exists and is non-zero.
\end{Lemma}

\begin{proof} Recall that for $p<q$, we have
$$\Theta(s,x,y)=x^2+y^2+(\ch s-\half\es y^2)^2.$$
We make the substitutions
\begin{equation}\label{subst}
v = \frac12 (1+\des (y^2 - 1)),\quad u=\es x,
\end{equation}
and find
\begin{equation}\label{subst1}
y^2=e^{-2s}(2v+\des-1),\qquad \ch s-\frac12\es y^2=\e-s(1-v),
\end{equation}
so that
\begin{equation}\label{subst2}
\Theta(s,x,y)=e^{-2s}(u^2+v^2)+1.
\end{equation}
Hence $Rf(s)$ equals
$$ e^{-s(\alpha + \beta +2)}\int_0^\infty \int_{(1-\des)/2} ^\infty
F(1+\de-s (u^2 + v^2)) (2v + \des -1)^{(\beta -1)/2} u^\alpha\, dv\, du.$$
Note that $\alpha+\beta+2=q$.
For $s$ sufficiently large negative we obtain a uniform bound
$$
\begin{aligned}
&e^{qs}|Rf(s)| \\
&\quad\le C \int_0^\infty
\int_{(1-\des)/2} ^\infty [1+e^{-2s}(u^2 + v^2)]^{-(\gamma +\rho)/2}
(2v + \des -1)^{(\beta -1)/2} u^\alpha\, dv\, du
\\
&\quad\le C e^{s(\gamma +\rho)}\int_0^\infty \int_{1/4} ^\infty
(u^2 + v^2)^{-(\gamma +\rho)/2} v^{(\beta -1)/2}
u^\alpha\, dv\, du < +\infty.
\end{aligned}
$$
It follows that we can apply Lebesgue's theorem to the limit
$$
\lim_{s\to-\infty}e^{(\rho_1-\gamma)s}\, Rf(s)
=\lim_{s\to-\infty} e^{(-\gamma-\rho+q)s}\, Rf(s)
.$$
Since
$r^{\gamma+\rho} F(r^2)$
allows a non-zero limit $c$ for $r\to\infty$,
it follows that
$$e^{-(\gamma+\rho)s} F(1+\de-s (u^2 + v^2))
\to c(u^2+v^2)^{-(\gamma+\rho)/2}$$
for $s\to-\infty$.
We conclude that $e^{(-\gamma-\rho+q)}\, Rf(s)$ tends to
$$ c\int_0^\infty \int_{1/2} ^\infty
 (u^2 + v^2)^{-(\gamma +\rho)/2} (2v-1)^{(\beta -1)/2}
u^\alpha\, dv\, du\neq 0
$$
as $s\to-\infty$.
\qed
\end{proof}

Proceeding with the proof of Theorem \ref{second thm},
we recall that $f=\xi_\lambda$. It follows from (\ref{bound xi})
and (\ref{limit xi}) that Lemma \ref{exponential limit lemma}
is applicable in the even case $n=2m$, and we conclude that
$ C:=\lim_{s\to -\infty} e^{(-\lambda+\rho_1)s}\, Rf(s) $ exists and is
non-zero. As $Rf(s)$ is known to be a combination
$C_1e^{(\lambda-\rho_1)s}+C_2e^{(-\lambda-\rho_1)s}$,
we conclude that $C_1=C$ and $C_2=0$.
Hence in this case
$$Rf(s)=Ce^{(\lambda-\rho_1)s}=Ce^{(\mu_\lambda-1)s}.$$

Finally, we consider the odd case, $\mu_\lambda = -n = -2m+1$,
where
$$f(k_\theta a_t)=\cos\theta\,
P_\lambda(\cosh^2t)(\cosh t)^{-\lambda-\rho-2m}.$$
Since $f$ is not $K$-invariant, Lemma
\ref{exponential limit lemma} is not directly applicable.
However, we shall adapt its proof.
It follows from (\ref{second cos}) that
$\cos\theta\ch t=\ch s-\frac12\e-s y^2$, and hence
$$Rf(s)=
\int_0^\infty \int_0^\infty x^\alpha y^\beta\,
(\ch s-\frac12e^sy^2)\,Q(s,x,y)
\,dx\,dy,
$$
where
$$Q(s,x,y)=P_\lambda(\Theta(s,x,y))
(\Theta(s,x,y))^{-\frac12(\lambda+\rho+1)-m}.$$
As in the proof of Lemma
\ref{exponential limit lemma} we perform
the substitutions (\ref{subst}). By application
of (\ref{subst1})-(\ref{subst2}) we see
$$e^{(q+1)s}Rf(s)=
\quad
\int_0^\infty \int_{(1-\des)/2}^\infty
u^\alpha
(2v + \des -1)^{(\beta -1)/2}(1-v)\,
\tilde Q(s,u,v)\,dv\,du,
$$
where
$$
\tilde Q(s,u,v)=P_\lambda(1+e^{-2s}(u^2 + v^2))
(1+e^{-2s}(u^2 + v^2))^{-\frac12(\lambda+\rho+1)-m}.
$$
This gives the domination for $s$ sufficiently large negative
$$
\begin{aligned}
&e^{(q+1)s}|Rf(s)| \\
&\ \le C \int_0^\infty
\int_{(1-\des)/2} ^\infty |1-v|\,[1+e^{-2s}(u^2 + v^2)]^{-(\lambda +\rho+1)/2}
(2v + \des -1)^{(\beta -1)/2}\,u^\alpha\,  dv\, du
\\
&\ \le C e^{s(\gamma +\rho+1)}\int_0^\infty \int_{1/4} ^\infty
(u^2 + v^2)^{-(\lambda +\rho+1)/2} v^{(\beta +1)/2}
u^\alpha\, dv\, du < +\infty.
\end{aligned}
$$
Again we can apply Lebesgue's theorem and obtain
$$ \lim_{s\to-\infty}e^{(1-\mu_\lambda)s} Rf(s) =
c\int_0^\infty \int_{1/2} ^\infty (1-v)(u^2 + v^2)^{-(\lambda +\rho+1)/2}u^\alpha
(2v -1)^{(\beta -1)/2}\, dv\, du,$$
with $c=\lim_{r\to \infty} r^{-(\lambda +\rho+1)}\phi_{n,m}(r^2)\neq 0$.

In order to prove that $Rf(s)= C e^{(\mu_\lambda - 1) s}$, with $C \ne 0$,
we argue as before. We only
need to establish that the following integral is non-zero:
$$I:=\int_0^\infty \int_{1/2} ^\infty
(1-v)(u^2 + v^2)^{-(\lambda +\rho+1)/2}u^\alpha (2v -1)^{(\beta -1)/2}\, dv\, du. $$
We rewrite this by setting $u=vx$ and $2v-1 = y$.
Up to a power of $2$ we obtain
\begin{equation}\label{integral not 0?}
\int_0^\infty (1+x^2)^{-(\lambda +\rho + 1)/2}x^\alpha dx \int_0^\infty
(1-y)y^{(\beta -1)/2}(y + 1)^{-(\lambda + \rho - \alpha)} dy,
\end{equation}
in which the first integral is clearly finite and non-zero.

We now use the formula
$$ \int_0^\infty y^{k-1} (1+y)^{-l} dy =
B(k,l-k)= \frac{\Gamma(k)\Gamma(l-k)}{\Gamma(l)}, $$
valid for $0<k<l$. It follows that for $0<k<l-1$
$$ \int_0^\infty (1-y)y^{k-1} (1+y)^{-l} dy =
(l-2k-1)\frac{\Gamma(k)\Gamma(l-k-1)}{\Gamma(l)}.$$
Hence the second integral
in (\ref{integral not 0?}) is zero
if and only if
$$(\lambda + \rho - \alpha)-(\beta +1)-1=0.$$
With the current values of $\alpha,\beta,\lambda$ and $\rho$,
we have $\lambda + \rho - \alpha-\beta-1=-n$, and hence
$I\neq0$.
This finishes the proof of Theorem \ref{second thm}.\qed

\bigskip

\end{document}